\documentclass[11pt,reqno]{amsart}

\usepackage{amssymb}
\usepackage{enumitem}
\usepackage{xcolor}
\usepackage{setspace}
\usepackage{graphicx}

\usepackage{url}
\makeatletter
\@namedef{subjclassname@2020}{%
  \textup{2020} Mathematics Subject Classification}
\makeatother

\usepackage[T1]{fontenc}

\newtheorem{theorem}{Theorem}[section]

\theoremstyle{definition}

\newtheorem*{theorem*}{Theorem}

\numberwithin{equation}{section}

\def\R{{\mathbb R}}
\def\e{{\rm e}}

\def\d{{\,\rm d}}

\begin{document}

\baselineskip=16.99pt


\title[]{Geometric aspects of the Harnack Inequality for a nonlocal heat equation}
\author[M. Dembny]{Mateusz Dembny}
\date{}
\address{
\begin{itemize}[label={}]
\item Faculty of Mathematics, Mechanics and Informatics\\ University of Warsaw\\
ul.\,Banacha 2, 02-097 Warsaw, Poland
\end{itemize}
}
\email{m.dembny@student.uw.edu.pl}

\author[M. Sierżęga]{Mikołaj Sierżęga}

\address{\begin{itemize}[label={}]
\item Faculty of Mathematics, Mechanics and Informatics\\ 
University of Warsaw\\
ul.\,Banacha 2, 02-097 Warsaw, Poland
\item Department of Mathematics Cornell University,
583 Malott Hall Ithaca, NY 14853 USA
\end{itemize}
}
\email{m.sierzega@uw.edu.pl, mikolaj.sierzega@cornell.edu}

\begin{abstract}
We establish a connection between a sharp double-sided Harnack bound for positive solutions of a fractional heat equation and the circular geometry in higher dimensions. The present work extends and generalizes the results obtained in the preceding paper.
\end{abstract}



\maketitle

\section{Introduction}
In \cite{DEMBNY_SIERZEGA_2023}, we were inspired by a natural problem, put forward in \cite{GAROFALO_2019}, of finding an appropriate extension for the Li-Yau inequality  to the context of nonlocal diffusion and in particular to the canonical model of one spatial dimension \textit{half-laplacian} heat equation

\begin{equation}\label{FH}
\partial_t u+(-\Delta )^\frac{1}{2}u=0 \quad \mbox{ on }\quad S_T,
\end{equation}
with 
\[
    (-\Delta)^\frac{1}{2}f(x)=-\frac{1}{\pi}P.V.\int_{\R}\frac{f(x)-f(y)}{|x-y|^{2}}\d y, \quad \mbox{ and  }\quad S_T=\R\times (0,T).
\]

Many previous results by various authors were based on the definition of a strong solution, which significantly restricted the class of functions satisfying the fractional heat equation. Let us recall some of these results.

Following \cite{BARRIOS_ET_AL_2014}, we will say that $u(x,t)$ is a \emph{strong solution} of the fractional heat equation \eqref{FH} in the strip $S_T$ if 
\begin{itemize}
    \item $\partial_t u\in C(\R\times (0,T))$,
    \item $ u\in C(\R\times [0,T))$,
    \item the equation \eqref{FH} is satisfied pointwise for every $(x,t)\in S_T$. 
\end{itemize}
The same definition, with obvious modifications, applies to the notion of the strong solution for the classical heat equation. The following elegant inequality has been proven.

\begin{theorem}[Thm.\,3.2  + Prop.\,3.3 in \cite{WEBER_ZACHER_2022}]
Let $u:S_T\mapsto (0,\infty)$ be a strong solution to the fractional heat equation \eqref{FH}. Then, the Li-Yau type inequality  
\[
    -(-\Delta )^\frac{1}{2}\ln u\geq -\frac{1}{ 2 t}
\]
holds in $S_T$. 
\end{theorem}

This inequality may be then employed to derive a Harnack bound. 
\begin{theorem}[Thm.\,5.2 in \cite{WEBER_ZACHER_2022}]
Let $0<t_1<t_2<\infty$ and $x_1,x_2\in \R$. If $u$ is a strong positive solution of \eqref{FH} on $\R\times [0,\infty)$, then 
\begin{equation}\label{ZW2}
    \frac{u(x_2,t_2)}{u(x_1,t_1)}\geq \sqrt{\frac{t_1}{t_2}}\e^{-C\left[1+\frac{|x_2-x_1|^2}{(t_2-t_1)^2}\right]},
\end{equation}
for some positive constant $C$.
\end{theorem}
For a broader formulation and additional results, the reader is referred to  \cite{WEBER_ZACHER_2022}, where the authors discuss in detail how their estimate differs from the classical Hadamard–Pini inequality. Notably, inequality \eqref{ZW2} does not become an identity when evaluated on the fractional heat kernel for any particular choice of the points $(x_i,t_i)$. Furthermore, given the polynomial decay of the heat kernel, it is natural to expect that a sharp Harnack inequality in this setting would exhibit polynomial, rather than exponential, decay. Finally, it would be advantageous to relax the assumption of continuity at the initial time, thereby accommodating more general initial data.

We have demonstrated, (see \cite{DEMBNY_SIERZEGA_2023}) how to derive an optimal fractional counterpart of the Hadamard-Pini bound in $\mathbb{R}\times (0,\infty)$, without the need to derive the Li–Yau inequality. The bound is sharp and independent of the initial condition. Moreover, our contribution concerned obtaining an \emph{unconditional} bound, which is to say, that apart from the solution being classical and positive, we do not impose any further restrictions on the spatial growth of solutions and of the initial data. A different strategy was necessary for the computations, relying on Widder’s representation theorem, which expresses the solution as a convolution of the heat kernel $k(x,t):=\frac{1}{\pi}\left(\frac{t}{t^2+|x|^2}\right)$, with the initial data, and involves estimating the kernel itself. We concluded with the following result.

\begin{theorem} 
Let $u$ be a positive classical solution of \eqref{FH} on $\R\times(0,T)$.  
Given $0<t_1,t_2<T$ and $x_1,x_2\in \R$, we have
\begin{equation}\label{FLY}
 \left(\frac{t_1}{t_2}\right) C_*\leq\frac{u(x_2,t_2)}{u(x_1,t_1)}\leq \left(\frac{t_1}{t_2}\right)C^*,
\end{equation}
with 
\[
C_*=\frac{\sqrt{\kappa_0}-(t_2-t_1)(t_2+t_1)-|x_2-x_1|^2} {\sqrt{\kappa_0}-(t_2-t_1)(t_2+t_1)+|x_2-x_1|^2},
\]
\[
C^*=\frac{\sqrt{\kappa_0}+(t_2-t_1)(t_2+t_1)+|x_2-x_1|^2} {\sqrt{\kappa_0}+(t_2-t_1)(t_2+t_1)-|x_2-x_1|^2},
\]
where 
\[
\kappa_0=\Big(\big|x_2-x_1\big|^2+\big|t_2-t_1\big|^2\Big)\Big(\big|x_2-x_1\big|^2+\big|t_2+t_1\big|^2\Big).
\]
Moreover, 
\[
C_*=\frac{x_1-x_*}{x_2-x_*}\quad \mbox{ and }\quad C^*=\frac{x_1-x^*}{x_2-x^*},
\]
where $x_*,x^*\in \R$ satisfy
\[
\left(\frac{t_1}{t_2}\right)C_*=\frac{k(x_2-x_*,t_2)}{k(x_1-x_*,t_1)}\quad \mbox{ and }\quad \frac{k(x_2-x^*,t_2)}{k(x_1-x^*,t_1)}=\left(\frac{t_1}{t_2}\right)C^*.
\]
\end{theorem}

We may now restate the assertion using the heat kernel and the notation introduced in the theorem.
\begin{equation}
\frac{k(x_2-x_*,t_2)}{k(x_1-x_*,t_1)}\leq\frac{u(x_2,t_2)}{u(x_1,t_1)}\leq \frac{k(x_2-x^*,t_2)}{k(x_1-x^*,t_1)}
\label{thesis}
\end{equation}

As we will see this rather complicated expression carries a geometric meaning whereby points $A,B, x_*, x^*$ lie on a semicircle in the $x,t$ plane. It will become clear later that the Harnack inequality, more precisely, the ratio appearing in it, can be interpreted as a distance in the hyperbolic metric. This observation allows us to fundamentally change the method of deriving the inequality, by adopting a geometric perspective, which in turn leads to the result we are interested in. As previously indicated, the aim of this work is to generalize the above result to the multidimensional setting.

\begin{figure}[htp]
    \centering
    \includegraphics[width=8cm]{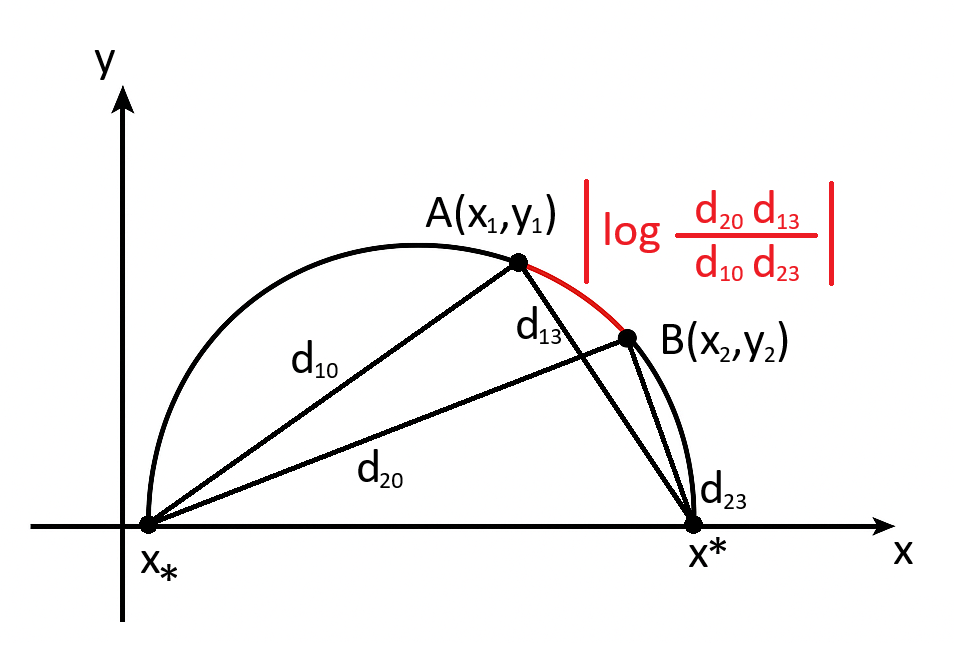}
    \caption{Semicircle}
    \label{semic}
\end{figure}

\section{Geometrical approach}

\subsection{Two dimentional inspiration}
We begin our discussion with an observation concerning distance in the hyperbolic metric in the Poincaré half-space $\mathbb{R}\times \mathbb{R}_+$. Given a semicircle centered on the horizontal axis, the hyperbolic distance, that is, the arc length between two points $A$ and $B$ is expressed by the absolute value of $\log (\frac{d_{20}d_{13}}{d_{10}d_{23}})$, where $d_{ab}$ denote the corresponding chords of the semicircle connecting $A$ and $B$ to the points where the semicircle intersects the horizontal axis (we denote them by $x_*$ and $x^*$). This idea is well illustrated in Figure \ref{semic}.

We will prove the following elementary fact.

\begin{theorem}
Consider two points \( A(x_1, t_1) \) and \( B(x_2, t_2) \) in \( \mathbb{R} \times \mathbb{R}_+ \). There exists a semicircle passing through \( A \) and \( B \), with its center lying on the line \( t = 0 \), such that its intersection points with the line \( t = 0 \) are exactly \( x_* \) and \( x^* \).

Let \( k(x, t; y) = \frac{1}{\pi}  \frac{t}{t^2 + (x - y)^2} \) denote the fundamental solution of the fractional heat equation in one spatial dimension, evaluated at the point \( (x, t) \), with initial data concentrated at the point \( y \). Then

$$
\frac{d_{20}d_{13}}{d_{10}d_{23}} = \left( \frac{k(A,x_*)}{k(B,x_*)}\frac{k(B,x^*)}{k(A,x^*)} \right)^{\frac{1}{2}},
$$

where $d_{ab}$ denotes the corresponding chords of the circle, as shown in Figure \ref{semic}.
\end{theorem}

\begin{proof}
Let us define the circle passing through two points \( A (x_1, t_1) \) and \( B (x_2, t_2) \), with its center lying on the line \( t = 0 \), by the equation
\[
(x - \alpha)^2 + t^2 = r^2.
\]
Substituting the coordinates of points \( A \) and \( B \) into this equation, we obtain the following system
\[
\begin{cases}
(x_1 - \alpha)^2 + t_1^2 = r^2, \\
(x_2 - \alpha)^2 + t_2^2 = r^2.
\end{cases}
\]
Subtracting the first equation from the second yields
\[
(x_2 - \alpha)^2 - (x_1 - \alpha)^2 + t_2^2 - t_1^2 = 0,
\]
which simplifies to
\[
x_2^2 - x_1^2 + 2\alpha(x_1 - x_2) + t_2^2 - t_1^2 = 0,
\]
and solving for \( \alpha \), we get
\[
\alpha = \frac{x_1^2 + t_1^2 - x_2^2 - t_2^2}{2(x_1 - x_2)}.
\]

Next, we compute the radius \( r \). The formula for the radius is given by
\[
r^2 = (x_1 - \alpha)^2 + t_1^2.
\]
Substituting the expression for \( \alpha \), we obtain
\[
r^2 = x_1^2 - 2x_1  \frac{x_1^2 + t_1^2 - x_2^2 - t_2^2}{2(x_1 - x_2)} + \left( \frac{x_1^2 + t_1^2 - x_2^2 - t_2^2}{2(x_1 - x_2)} \right)^2 + t_1^2,
\]
and therefore,
\[
r = \sqrt{x_1^2 - x_1  \frac{x_1^2 + t_1^2 - x_2^2 - t_2^2}{x_1 - x_2} + \left( \frac{x_1^2 + t_1^2 - x_2^2 - t_2^2}{2(x_1 - x_2)} \right)^2 + t_1^2}.
\]

To find the points of intersection of the circle with the line \( t = 0 \), we solve
\[
(x - \alpha)^2 = r^2 \quad \Rightarrow \quad (x - \alpha - r)(x - \alpha + r) = 0,
\]
which yields
\[
x_* = \alpha - r, \quad x^* = \alpha + r.
\]

We compute the Euclidean lengths of the chords, as illustrated in Figure \ref{semic}
\[
d_{20} = \sqrt{t_2^2 + (x_2 - x_*)^2}, \quad
d_{10} = \sqrt{t_1^2 + (x_1 - x_*)^2},
\]
\[
d_{23} = \sqrt{t_2^2 + (x_2 - x^*)^2}, \quad
d_{13} = \sqrt{t_1^2 + (x_1 - x^*)^2}.
\]

We now consider the logarithmic input
\[
\frac{d_{20}  d_{13}}{d_{10}  d_{23}} = 
\frac{\sqrt{t_2^2 + (x_2 - x_*)^2}  \sqrt{t_1^2 + (x_1 - x^*)^2}}{\sqrt{t_1^2 + (x_1 - x_*)^2}  \sqrt{t_2^2 + (x_2 - x^*)^2}}.
\]

On the other hand, consider the fundamental solution \( k(x, t; y) \) of the one-dimensional fractional heat equation, given by
\[
k(x, t; y) = \frac{1}{\pi}  \frac{t}{t^2 + (x - y)^2}.
\]
We compute the kernel ratios
\[
\frac{k(B, x_*)}{k(A, x_*)} = \frac{t_2}{t_1}  \frac{t_1^2 + (x_1 - x_*)^2}{t_2^2 + (x_2 - x_*)^2},
\quad
\frac{k(B, x^*)}{k(A, x^*)} = \frac{t_2}{t_1}  \frac{t_1^2 + (x_1 - x^*)^2}{t_2^2 + (x_2 - x^*)^2}.
\]

We now relate the geometric and analytic expressions
\[
\frac{d_{20} d_{13}}{d_{10}  d_{23}} =
\sqrt{
\frac{t_2^2 + (x_2 - x_*)^2}{t_1^2 + (x_1 - x_*)^2}
\frac{t_1^2 + (x_1 - x^*)^2}{t_2^2 + (x_2 - x^*)^2}
}=
\]
\[
=
\sqrt{
\frac{k(A, x_*)}{k(B, x_*)}
\frac{k(B, x^*)}{k(A, x^*)}
}
=
\left(
\frac{k(A, x_*)}{k(B, x_*)}
\frac{k(B, x^*)}{k(A, x^*)}
\right)^{1/2}.
\]

This completes the derivation of the geometric expression in terms of the Poisson kernel.
\end{proof}

As observed, the argument of the logarithm, describing the arc length between two points
$A$ and $B$ on a circle passing through them, is directly related to the ratio of the fundamental solutions evaluated at $A$ and $B$.

\subsection{Plane in three dimentional space}

Now, as shown in Figure \ref{circplane}, given four points in $\mathbb{R}^2\times \mathbb{R}_+$. We denote them by $A(x_1, y_1, t_1)$, $B(x_2, y_2, t_2)$ and their respective projections onto the plane $\lbrace (x,y,t)\;|\; t=0 \rbrace$, i.e. $A'(x_1, y_1, 0)$, $B' = (x_2, y_2, 0)$.

We will seek the desired circle in the plane that contains four given points. The idea is to exploit the observation mentioned before concerning the Poincaré metric in two dimensions, applied within an appropriate subspace. We will then make use of existing formulas to compute the arc length, which is directly related to the fundamental solution of the fractional heat equation. 

The plane is given by

$$
P=\lbrace (x,y,t)\;|\; \alpha x+\beta y+\gamma t+\delta=0 \rbrace.
$$

Since the plane is parallel to the vertical axis, it means $\gamma=0$. Therefore

$$
P=\lbrace (x,y,t)\;|\; \alpha x+\beta y+\delta =0 \rbrace.
$$

We shall now determine the constants $\alpha$,$\beta$ and $\delta$.

Since, both $A$ and $B$ belong to the semicircle, they have to satisfy

\begin{equation}
\begin{cases}

\alpha x_1+\beta y_1+\delta =0 \\
\alpha x_2+\beta y_2+\delta=0.

\end{cases}
    \label{plane}
\end{equation}

And by subtracting the second from the first equation, we obtain the following

$$
\alpha(x_1-x_2)+\beta (y_1-y_2)=0.
$$
Let $c$ be an arbitrary constant. Then, setting $\alpha=c(y_2-y_1)$ and $\beta=c(x_1-x_2)$, we observe that these choices satisfy the equation above. Consequently, $\delta$ is given by $-c(y_2-y_1)x_1-c(x_1-x_2)$.

Finally, by substituting the computed constants into the plane equation, we obtain the expression
$$
P= \lbrace (x,y,t)\;|\; (y_2-y_1)x+(x_1-x_2)y=(y_2-y_1)x_1+(x_1-x_2)y_2, \; t\in \mathbb{R_+} \rbrace,
$$
which is 

\begin{equation}\label{PL}
 P= \lbrace (x,y,t)\;|\; (y_2-y_1)(x-x_1)=(x_2-x_1)(y-y_1), \; t\in \mathbb{R_+} \rbrace .
\end{equation}

\begin{figure}[htp]
    \centering
    \includegraphics[width=12cm]{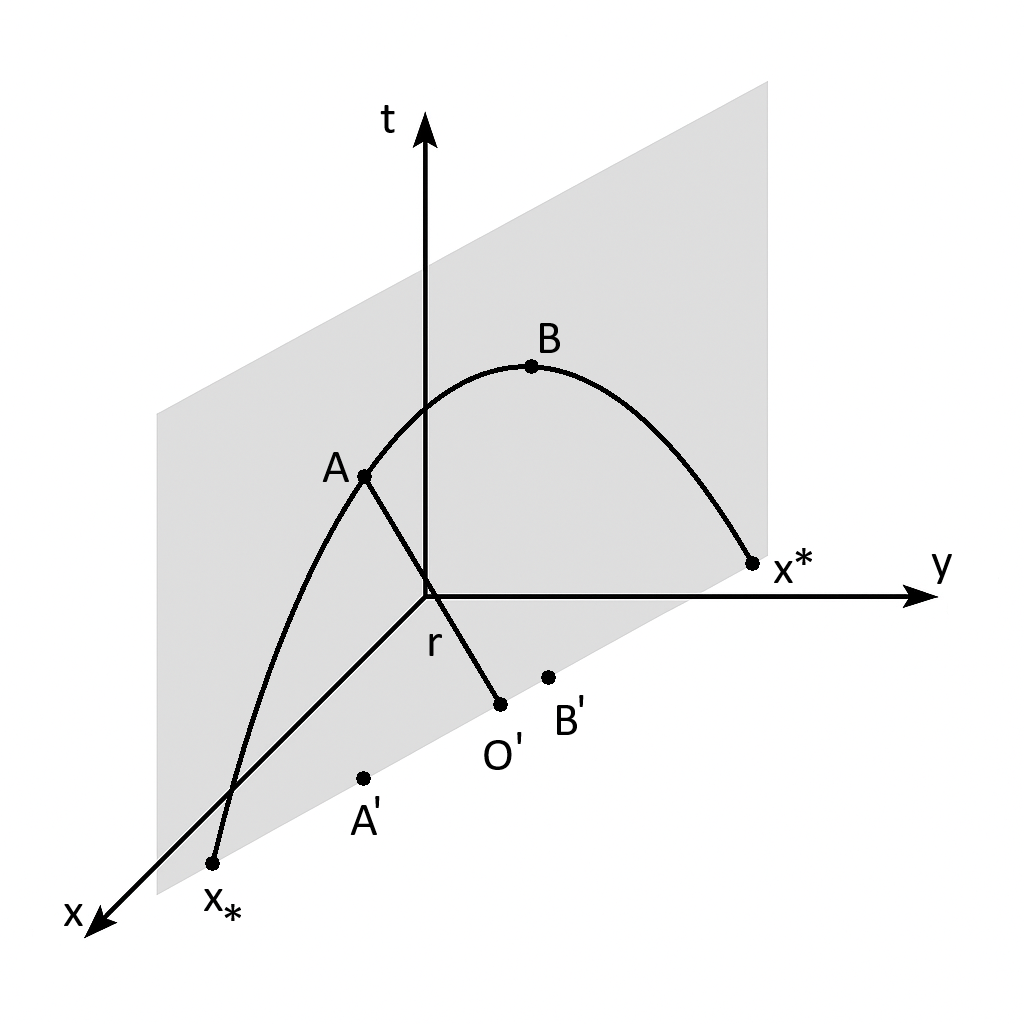}
    \caption{Semicircle in the special plane}
    \label{circplane}
\end{figure}

\subsection{Circle radius and center}
To determine the center and radius of the desired circle, we employ well-known formulas that lead to systems of equations. 

The plane is defined by the equation
\begin{equation}
    (x_2 - x_1)(y - y_1) - (y_2 - y_1)(x - x_1) = 0.
    \label{eq:plane}
\end{equation}

Let the point \( O' = (x_0, y_0, 0) \) be the center of the circle. It has to satisfy two conditions:
it lies in the plane and is equidistant from A and B. Therefore, we express these conditions as follows

\begin{equation}
    \begin{cases}
    (x_2 - x_1)(y_0 - y_1) - (y_2 - y_1)(x_0 - x_1) = 0 \\
     \| O' - A \|^2 = \| O' - B \|^2,
\end{cases}
\label{equidist}
\end{equation}

which expands to

$$
\begin{cases}
    (x_2 - x_1)(y_0 - y_1) - (y_2 - y_1)(x_0 - x_1) = 0 \\
    (x_0 - x_1)^2 + (y_0 - y_1)^2 + t_1^2 = (x_0 - x_2)^2 + (y_0 - y_2)^2 + t_2^2 
\end{cases}
$$

or equivalently

$$
\begin{cases}
    (x_2 - x_1)(y_0 - y_1) - (y_2 - y_1)(x_0 - x_1) = 0 \\
    -2x_0(x_1 - x_2) - 2y_0(y_1 - y_2) = t_2^2 - t_1^2 - (x_1^2 - x_2^2) - (y_1^2 - y_2^2).
\end{cases}
$$

Now we need to solve this system of equations. Once we do this, finding radius r does not pose any difficulty. When the center \( O' = (x_0, y_0, 0) \) is found, the radius \( r \) can be computed from either point \( A \) or \( B \) as
$$
    r = \sqrt{(x_0 - x_1)^2 + (y_0 - y_1)^2 + t_1^2} \quad \mbox{or}\quad  r = \sqrt{(x_0 - x_2)^2 + (y_0 - y_2)^2 + t_2^2}.
$$

Let us denote the following constants to simplify the system:
$a=x_2-x_1$, $b=y_2-y_1$ and $E=t_2^2-t_1^2-(x_1-x_2)^2-(y_1-y_2)^2$. At that point, the system becomes

$$
\begin{cases}
    ay_0-bx_0 = ay_1-bx_1 \\
    -2ax_0-2by_0=E.
\end{cases}
$$

From the first equation, we obtain
\[
a y_0 = b x_0 + a y_1 - b x_1 \Rightarrow 
y_0 = \frac{b x_0 + a y_1 - b x_1}{a}.
\]

Substituting into the second equation yields

$$
-2a x_0 - 2b \cdot \left( \frac{b x_0 + a y_1 - b x_1}{a} \right) = E ,
$$
which is
$$
-2a x_0 - \frac{2b^2 x_0 + 2ab y_1 - 2b^2 x_1}{a} = E.
$$

We combine terms
$$
\left(-2a - \frac{2b^2}{a}\right)x_0 = E + 2b y_1 - \frac{2b^2 x_1}{a},
$$

and multiply both sides by $a$
$$
(-2a^2 - 2b^2)x_0 = aE + 2ab y_1 - 2b^2 x_1.
$$

We therefore arrive at the following expression for $x_0$

$$
x_0 = \frac{ -a E - 2ab y_1 + 2b^2 x_1 }{ 2(a^2 + b^2) }=
$$

$$
=\frac{ -a (t_2^2-t_1^2-(x_1-x_2)^2-(y_1-y_2)^2) - 2ab y_1 + 2b^2 x_1 }{ 2(a^2 + b^2) }=
$$

$$
\frac{ -(x_2-x_1) (t_2^2-t_1^2-(x_1-x_2)^2-(y_1-y_2)^2) - 2(x_2-x_1)(y_2-y_1) y_1 }{ 2((x_2-x_1)^2 + (y_2-y_1)^2) }+
$$

$$
+\frac{ 2(y_2-y_1)^2 x_1 }{ 2((x_2-x_1)^2 + (y_2-y_1)^2) }=
$$

$$
=\frac{ -(x_2-x_1) (t_2^2-t_1^2-(x_1-x_2)^2-(y_1-y_2)^2) }{ 2((x_2-x_1)^2 + (y_2-y_1)^2) } +
$$

$$
\frac{ - 2(x_2-x_1)(y_2-y_1) y_1 - 2(x_2-x_1)^2 x_1 }{ 2((x_2-x_1)^2 + (y_2-y_1)^2) } +x_1.
$$

Since the algebraic expressions become very large, we will stick to the previously introduced notation. 

\begin{equation}
y_0 = \frac{ b x_0 + a y_1 - b x_1 }{ a }= \frac{b}{a}(x_0-x_1) + y_1.
\end{equation}

The formulas are particularly convenient, as they enable the differences to be written directly. We have

$$
x_0 - x_1 = \frac{ -(x_2-x_1) (t_2^2-t_1^2-(x_1-x_2)^2-(y_1-y_2)^2) }{ 2((x_2-x_1)^2 + (y_2-y_1)^2) }
$$

$$
+ \frac{ - 2(x_2-x_1)(y_2-y_1) y_1 - 2(x_2-x_1)^2 x_1 }{ 2((x_2-x_1)^2 + (y_2-y_1)^2) }.
$$

$$
y_0 - y_1 = \frac{b}{a}(x_0-x_1).
$$

And the radius

\begin{equation}
r = \sqrt{(x_0 - x_1)^2 + (y_0 - y_1)^2 + t_1^2}
\label{radius}
\end{equation}

The points of intersection of the circle with the plane $t=0$ can be easily computed. It suffices to move a distance $r$ from the center of the circle along the diameter. This yields

\begin{equation}
    x_* = O' - r \frac{A'B'}{|A'B'|}, \quad x^* = O' + r \frac{A'B'}{|A'B'|}
    \label{xstar}
\end{equation}

\subsection{Ratio of the chords}

Following \cite{GAROFALO_2019}, we introduce the Poisson kernel as the fundamental solution to the heat equation involving the half-laplacian. Below, we present the formula that will be used to relate it to the arc length of the corresponding circle.

\begin{equation}
P(\vec{x},t) = c_n \frac{t}{(t^2 + |\vec{x}|^2)^{\frac{n+1}{2}}}, \quad \mbox{ with }\quad c_n = \frac{\Gamma (\frac{n+1}{2})}{\pi^{\frac{n+1}{2}}}.
    \label{poisson}
\end{equation}

In the context of our current problem, $n=2$. Hence $\Gamma (\frac{3}{2})=\frac{\sqrt{\pi}}{2}$ and $c_2=\frac{\frac{\sqrt{\pi}}{2}}{\pi \sqrt{\pi}}=\frac{1}{2\pi}$. Let $K(\vec{x},t;\vec{y})$ represent the Poisson kernel with a spatial shift.

$$
K(x,y,t;u,v) = \frac{1}{2\pi} \frac{t}{(t^2+|(x,y,0)-(u,v,0)|^2)^{\frac{3}{2}}}.
$$

Following the relation depicted in Figure \ref{semic}, we focus on the chord-based expression and relate it to the Poisson kernel. The chords are given by formulas below, which follow directly from the Pythagorean theorem

$$
D_{20}=\sqrt{t_2^2+|B'-x_*|^2}, \quad D_{13}=\sqrt{t_1^2+|A'-x^*|^2},
$$
$$
D_{10}=\sqrt{t_1^2+|A'-x_*|^2}, \quad D_{23}=\sqrt{t_2^2+|B'-x^*|^2}.
$$

Then, the input of the logarithm is given by the following
$$
\frac{D_{20}D_{13}}{D_{10}D_{23}}=\frac{\sqrt{t_2^2+|B'-x_*|^2}\sqrt{t_1^2+|A'-x^*|^2}}{\sqrt{t_1^2+|A'-x_*|^2}\sqrt{t_2^2+|B'-x^*|^2}}.
$$
On the other hand

$$
\frac{K(B,x_*)}{K(A,x_*)}=\frac{t_2}{t_1} \left( \frac{t_1^2+|A'-x_*|^2}{t_2^2+|B'-x_*|^2}\right)^{\frac{3}{2}}, \quad \frac{K(B,x^*)}{K(A,x^*)}=\frac{t_2}{t_1} \left( \frac{t_1^2+|A'-x^*|^2}{t_2^2+|B'-x^*|^2}\right)^{\frac{3}{2}}.
$$

We observe a similarity between the formulas, and thus we may combine them.

$$
\frac{D_{20}D_{13}}{D_{10}D_{23}} = \sqrt{\frac{t_2^2+|B'-x_*|^2}{t_1^2+|A'-x_*|^2}\frac{t_1^2+|A'-x^*|^2}{t_2^2+|B'-x^*|^2}}=
$$

$$
=\sqrt{\left(\frac{t_2}{t_1}\right)^{2/3}\left(\frac{K(A,x_*)}{K(B,x_*)}\right)^{2/3}\frac{t_1^2+|A'-x^*|^2}{t_2^2+|B'-x^*|^2}}=
$$

$$
=\sqrt{\left(\frac{K(A,x_*)}{K(B,x_*)}\right)^{2/3}\left(\frac{K(B,x^*)}{K(A,x^*)}\right)^{2/3}}=\left( \frac{K(A,x_*)}{K(B,x_*)}\frac{K(B,x^*)}{K(A,x^*)}\right)^{\frac{1}{3}}
$$

This crucial geometric insight provides the foundation for extending the results of \cite{DEMBNY_SIERZEGA_2023}, by avoiding the technical difficulties typically encountered in higher-dimensional settings. In the next section, we demonstrate the corresponding relation in $n$ spatial dimensions and draw a comparison with equation \eqref{thesis}.

\section{The main formula}

Here, we formulate our reasoning in the most general setting. Now the four points are $A (x_1, \dots, x_n, t_A)$, $B (y_1, \dots,  y_n, t_B)$, $A' (x_1, \dots, x_n, 0)$ and \\ $B' (y_1, \dots,  y_n, 0)$.

Systems \eqref{plane} and \eqref{equidist} are, and will remain in higher dimensions, linear systems with respect to the coordinates of $A,B$ and $O'$. By construction, these systems admit solutions, and one may therefore, in general, apply standard formulas for solving linear systems. For the sake of elegance and clarity, the authors omit these computations and proceed directly to the main formula. 

Similarly, obtaining expressions for $r, x_*$ and $x^*$ in higher dimensions is straightforward, as formulas \eqref{radius} and \eqref{xstar} extend naturally to the $n$-dimensional setting.

By employing analogous expressions for $D_{ab}$ we have

$$
D_{20}=\sqrt{t_B^2+|B'-x_*|^2}, \quad D_{13}=\sqrt{t_A^2+|A'-x^*|^2},
$$
$$
D_{10}=\sqrt{t_A^2+|A'-x_*|^2}, \quad D_{23}=\sqrt{t_B^2+|B'-x^*|^2}.
$$

At this point, we recall the formula for the Poisson kernel \eqref{poisson}
and its variant

$$
K(\vec{x},t;\vec{y}) = c_n \frac{t}{(t^2 + |\vec{x}-\vec{y}|^2)^{\frac{n+1}{2}}}.
$$

By repeating the calculations carried out in the preceding section, we obtain

$$
\frac{K(B,x_*)}{K(A,x_*)}=\frac{c_n \frac{t_B}{(t_B^2 + |x_*-B'|^2)^{\frac{n+1}{2}}}}{c_n \frac{t_A}{(t_A^2 + |x_*-A'|^2)^{\frac{n+1}{2}}}}=\frac{t_B}{t_A} \left( \frac{t_A^2 + |x_*-A'|^2}{t_B^2 + |x_*-B'|^2}\right)^{\frac{n+1}{2}},
$$

and thus

$$
\frac{K(B,x^*)}{K(A,x^*)}=\frac{t_B}{t_A} \left( \frac{t_A^2 + |x^*-A'|^2}{t_B^2 + |x^*-B'|^2}\right)^{\frac{n+1}{2}}.
$$

Finally, we are in a position to carry out the final computation that connects the above results. This constitutes the central point of the present work.

$$
\frac{D_{20}D_{13}}{D_{10}D_{23}} = \sqrt{\frac{t_B^2+|B'-x_*|^2}{t_A^2+|A'-x_*|^2}\frac{t_A^2+|A'-x^*|^2}{t_B^2+|B'-x^*|^2}}=
$$

$$
=\sqrt{\left(\frac{t_B}{t_A}\right)^{2/n+1}\left(\frac{K(A,x_*)}{K(B,x_*)}\right)^{2/n+1}\frac{t_1^2+|A'-x^*|^2}{t_2^2+|B'-x^*|^2}}=
$$
$$
=\sqrt{\left(\frac{K(A,x_*)}{K(B,x_*)}\right)^{2/n+1}\left(\frac{K(B,x^*)}{K(A,x^*)}\right)^{2/n+1}}=
$$

$$
=\left( \frac{K(A,x_*)}{K(B,x_*)}\frac{K(B,x^*)}{K(A,x^*)}\right)^{\frac{1}{n+1}}.
$$

It has been established that

\begin{theorem}
    Let $A,B\in \mathbb{R}^n\times \mathbb{R}_+$ be a couple of points, and let $x_*, x^*$ denote the points of intersection of the circle lying in the plane containing A and B, passing through A and B, with the plane $\lbrace (x,t)\;|\;x\in \mathbb{R}^n, \; t=0 \rbrace$. Furthermore, let $K(\vec{x},t;\vec{y})$ denote the Poisson kernel, i.e.

    $$
    K(\vec{x},t;\vec{y}) = c_n \frac{t}{(t^2 + |\vec{x}-\vec{y}|^2)^{\frac{n+1}{2}}}.
    $$

    Then the following formula holds

\begin{equation}
    \frac{D_{20}D_{13}}{D_{10}D_{23}} = \left( \frac{K(A,x_*)}{K(B,x_*)}\frac{K(B,x^*)}{K(A,x^*)}\right)^{\frac{1}{n+1}},
    \label{maineq}
\end{equation}

where $D_{ab}$ are the chords described by the expressions

$$
D_{20}=\sqrt{t_B^2+|B'-x_*|^2}, \quad D_{13}=\sqrt{t_A^2+|A'-x^*|^2},
$$

$$
D_{10}=\sqrt{t_A^2+|A'-x_*|^2}, \quad D_{23}=\sqrt{t_B^2+|B'-x^*|^2}.
$$
\end{theorem}

\section{Discussion}
The classical parabolic Harnack estimate, originally established independently by Hadamard and Pini \cite{HADAMARD_1954,PINI_1954} was initially obtained through a suitable integration of the Li–Yau differential inequality (see \cite{HAMILTON_2011,LI_YAU_1986}). The integration was performed along the line segment connecting two points. It becomes apparent that a choice of an appropriate curve connecting two points is a significant aspect in the theory of Harnack inequalities.

In \cite{DEMBNY_SIERZEGA_2023}, the Harnack inequality was established by expressing the solution via the Widder representation theorem, as the convolution of the fundamental solution with the initial condition. We managed to derive the Harnack inequality without using the Li–Yau inequality; however, this required paying the price of extensive calculations. The proof relied on estimating the Poisson kernel. 

In that work, the estimation led to a formula that admitted a geometric interpretation, in which both the radius and the center of a certain circle could be identified. This observation motivated the conjecture that, in higher dimensions, a similar connection with circular geometry should persist and the curve we seek is an arc of a circle.

We postulated the existence of a circle determined by the points $A,B,A'$, $B'$ and this led to the emergence of a new formula, which, as it turns out, aligns with the authors’ expectations. Whereas the earlier work focused on a technical estimation, the present work provides a geometric interpretation. Grasping it, allowed us to bypass the extensive computations and focus on well-established formulas.

\subsection*{Acknowledgements}
This research was partially supported by the Polish National Science Center grant SONATA BIS no.\,2020/38/E/ST1/00596.

The second author also acknowledges the support of the NAWA Bekker Scholarship Programme BPN/BEK/2021/1/00277.   

All Figures were generated using ChatGPT.

\bibliographystyle{siam} 
\bibliography{biblio} 
 \end{document}